\documentclass[12pt,reqno]{amsart} 
\usepackage[T1]{fontenc}
\usepackage{amsfonts}
\usepackage{amssymb}
\usepackage{mathrsfs}
\usepackage[latin1]{inputenc}
\usepackage{amsmath}
\usepackage{paralist}
\usepackage[english]{babel}
\usepackage{setspace}
\usepackage{bbm}

\newtheorem{theorem}{Theorem}[section]
\newtheorem{lemma}[theorem]{Lemma}
\newtheorem{definition}[theorem]{Definition}
\newtheorem{remark}[theorem]{Remark}
\newtheorem{prop}[theorem]{Proposition}

\newtheorem{corollary}[theorem]{Corollary}

\numberwithin{equation}{section}

\newcommand{\cS}{{\mathcal S}}

\begin{document}
	
	\title{Topologizability and related properties of the iterates of composition operators in Gelfand-Shilov classes}
	
	\author{Angela\,A. Albanese, H\'ector Ariza}


\address{ Angela A. Albanese\\
	Dipartimento di Matematica e Fisica
	``E. De Giorgi''\\
	Universit\`a del Salento- C.P.193\\
	I-73100 Lecce, Italy}
\email{angela.albanese@unisalento.it}

\address{H\'ector Ariza\\
Escuela T\'ecnica Superior de Arquitectura\\
 Universitat Polit\'ecnica de Val\'encia\\ 46022 Valencia\\ Spain}
\email{harirem@upvnet.upv.es}

\begin{abstract}
	We analyse the behaviour of the iterates of composition operators defined by polynomials acting on global classes of ultradifferentiable functions of Beurling type which are invariant under the Fourier transform. In particular, we determine the polynomials $\psi$ for which the sequence of iterates of the composition operator $C_\psi$ is topologizable ($m$-topologizable) acting on certain Gelfand-Shilov spaces defined by mean of Braun-Meise-Taylor weights. We prove that the composition operators $C_\psi$ with $\psi$  a polynomial of degree greater than one are always topologizable in certain settings involving Gelfand-Shilov spaces, just like in the Schwartz space. Unlike in the Schwartz space setting, composition operators $C_\psi$ associated with polynomials $\psi$ are not always $m-$topologizable. We also deal with the composition operators $C_\psi$ with $\psi$ being an affine function acting on $\cS_{\omega}(\mathbb{R})$ and find a complete characterization of topologizability and $m-$topologizability.
	\end{abstract}

\maketitle

\markboth{A.\,A. Albanese, H. Ariza}%
{\MakeUppercase{Topologizability and related properties }}

\section{Introduction}

The operators on topological vector spaces of  functions have been extensively studied for the last several decades, given their fundamental role in various branches of mathematical analysis.  In the theory of functional spaces, the most extensively studied operators are the multiplication operators and the composition operators.

The composition operators have been investigated on spaces of (ultra)differentia\-ble functions defined on an open set $\Omega \subseteq \mathbb{R}^d$. These investigations have considered various directions, with a main question being the characterization of conditions under which such operators are well-defined and continuous. It is a well-known fact that for the space $C^\infty(\Omega)$, the composition operator is both well-defined and continuous. Regarding the Schwartz space $\mathcal{S}(\mathbb{R})$ of rapidly decreasing smooth functions, the conditions for which the composition operator is well-defined were characterized in \cite{galbis}. Further properties of these operators acting in $\mathcal{S}(\mathbb{R})$ have also been investigated in \cite{akj,fgj,advances}.

The study of composition operators in Gelfand-Shilov spaces ${\mathcal S}_\omega({\mathbb R})$ has begun  in \cite{jmaa, RACSAM}.
The aim of this paper is to continue the research initiated in \cite{jmaa, RACSAM}.
The Gelfand-Shilov spaces ${\mathcal S}_\omega({\mathbb R})$ here considered are defined by mean of
Braun-Meise-Taylor weights, as presented in \cite{bmt}. These spaces, ${\mathcal S}_\omega({\mathbb R})$, are versions of the Schwartz space in an ultradifferentiable setting, retaining invariance under the Fourier transform. First introduced in \cite{bjorck}, the most relevant examples correspond to the classical Gelfand-Shilov spaces $\Sigma_s({\mathbb R}).$
The study of these classes and of operators acting on them is a very   active research area (see, for instance, \cite{am,ajm,asensio_jornet,capiello,cprt,debrouwere} and the references therein).

The Gelfand-Shilov spaces are not invariant under composition with polynomials of degree greater than one as noted in \cite[Theorem 3.9]{jmaa}. However, for every sub-additive weight function $\omega$, a new weight $\sigma$
can be identified such that $f\circ \psi\in {\mathcal S}_\sigma({\mathbb R})$ 
whenever $f\in {\mathcal S}_\omega({\mathbb R})$ and $\psi$
is a polynomial (see \cite[Theorem 4.4]{jmaa} for the non-trivial case where $\psi$ has a degree greater than one; otherwise, $\sigma=\omega$ is a valid choice). Precisely, for any  sub-additive weight $\omega$  and $\sigma:=\omega(\bullet^{\frac{1}{2}})$, we have $f\circ \psi\in {\mathcal S}_\sigma({\mathbb R})$ 
for every $f\in {\mathcal S}_\omega({\mathbb R})$,  with $\psi$  a fixed
polynomial. In such a case, the composition operator $C_\psi: \mathcal{S}_{\omega}(\mathbb{R})\to \mathcal{S}_{\sigma}(\mathbb{R})$ is clearly well-defined and  continuous. Given that the choice of 
$\sigma$ is independent of the polynomial 
$\psi$, it is then interesting to examine further the dynamics of the operator $C_\psi\colon {\mathcal S}_\omega({\mathbb R})\to {\mathcal S}_\sigma({\mathbb R})$, $f\mapsto f\circ \psi$. In \cite{jmaa} it was also showed
that the weight $\sigma=\omega(\bullet^{\frac{1}{2}})$ is optimal for studying iterations and dynamics of composition
operators associated with polynomials of degree greater than one (see \cite[Corollary 3.8]{jmaa}). In \cite{RACSAM} the authors characterized the polynomials $\psi$ of degree greater than one for which $C_\psi:\mathcal{S}_{\omega}(\mathbb{R})\to \mathcal{S}_{\sigma}(\mathbb{R})$ is equicontinuous, where $\omega$ is a sub-additive weight and $\sigma=\omega(\bullet^{\frac{1}{a}})$, $a>2$. In particular, they proved that if $\psi$ is a polynomial of degree greater than or equal to two, then the equicontinuity of the sequence of the iterates of the composition operator $C_\psi:\mathcal{S}_{\omega}(\mathbb{R})\to \mathcal{S}_{\sigma}(\mathbb{R})$ is equivalent to the lack of fixed points of the polynomial $\psi$.

The paper is organized as follows.
In Sections \ref{sec2} and \ref{sec3} we  study what happens with the intermediate properties of topologizability and $m-$topologizability of iterates of composition operators $C_\psi: \mathcal{S}_{\omega}(\mathbb{R})\to \mathcal{S}_{\sigma}(\mathbb{R})$, where $\psi$ is a polynomial. In Section 2 the case of polynomials of degree greater than one and $\sigma=\omega(\bullet^{\frac{1}{a}})$, with $a>2$, is analyzed. In Section \ref{sec3}
the case of polynomial $\psi$ of degre one 
and $a=1$ (i.e. $\sigma=\omega$) is studied.

From now on, $\psi_m = \psi\circ\ldots\circ\psi$ denotes the $m$-th iteration of $\psi.$

\medskip
\subsection{Gelfand-Shilov function spaces}\label{sec1.1}
We begin by describing the Gelfand-Shilov function spaces defined by mean  of non-quasianalytic weight functions $\omega$ in the sense of Braun, Meise and Taylor \cite{bmt}.

\begin{definition}\label{def:weight} A continuous increasing function $\omega :[0,\infty [\longrightarrow [0,\infty [$ is called a {\it non quasi-analytic weight} if it satisfies:
	\begin{itemize}
		\item[$(\alpha)$] there exists $L\geq 1$ such that $\omega (2t) \leq L(\omega (t)+1)$ for
		all $t\geq 0$,
		\item[$(\beta)$] $\displaystyle\int_{0}^{\infty}\frac{\omega (t)}{1+t^{2}}\ dt < \infty $,
		\item[$(\gamma)$] $\log(1+t^{2})=o(\omega (t))$ as t tends to $\infty$,
		\item[$(\delta)$] $\varphi_\omega :t\rightarrow \omega (e^{t})$ is convex.
	\end{itemize}
\end{definition}

Observe that condition $(\alpha)$ implies that $\omega(t_1+t_2)\leq L (\omega(t_1)+\omega(t_2)+1)$, for every $t_1, t_2\geq 0$ (see, for example \cite[Remark 2.2(1)]{paley}). If the weight $\omega$ satisfies the strongest  condition $\omega(t_1+t_2)\leq \omega(t_1) + \omega(t_2)$, for every $t_1, t_2\geq 0$, then $\omega$ is called a sub-additive weight.
In the main results of the paper,  we assume the weights to be sub-additive
or equivalent to a sub-additive one. The main examples are $\omega(t) = t^{\frac{1}{d}}$, $d>1,$ or $\omega(t) = \max\left(0, \log^p t\right)$, $p > 1.$ We recall that two weights $\omega$ and $\sigma$ are said to be equivalent if there exist $A, B > 0$ such that $\omega(t) \leq A\left(1 + \sigma(t)\right)$ and $\sigma(t) \leq B\left(1 + \omega(t)\right)$ for all $t\geq 0.$ 
\par\medskip
The {\it Young conjugate} $\varphi_\omega ^{*}:[0,\infty [ \longrightarrow {\mathbb R}$  of
$\varphi_\omega$ is defined by
$$
\varphi_\omega^{*}(s):=\sup\{ st-\varphi_\omega (t):\ t\geq 0\},\ s\geq 0.
$$ Then $\varphi_\omega^\ast$ is convex, $\varphi_\omega^\ast(s)/s$ is increasing and $\displaystyle\lim_{s\to \infty}\frac{\varphi_\omega^\ast(s)}{s} = +\infty.$ 

A weight function  $\omega$ is said to be a {\it strong weight} if
\begin{itemize}
	\item[$(\varepsilon)$]
	there exists a constant $C \geq 1$ such that for all $y > 0$ the following inequality holds
	\begin{equation}
		\int_1^\infty \frac{\omega(yt)}{t^2}\ dt \leq C\omega(y) + C.
	\end{equation}
\end{itemize}

\begin{definition} Let $\omega$ be a weight function. The Gelfand-Shilov space of Beurling type $\mathcal{S}_{\omega}(\mathbb{R})$ consists of those functions $f\in C^\infty({\mathbb R})$ with the property that
	$$
	p_{\omega, \lambda}(f):= \sup_{x\in {\mathbb R}}\sup_{n,q\in {\mathbb N}_0}\left(1+|x|\right)^q|f^{(n)}(x)|\exp\left(-\lambda\varphi_\omega^\ast\left(\frac{n+q}{\lambda}\right)\right) < \infty
	$$ 
	for every $\lambda > 0.$
\end{definition}
\par\medskip 

Condition $(\beta)$ guarantees the existence of non-trivial compactly supported functions $f\in \mathcal{S}_{\omega}(\mathbb{R}),$ which plays an important role in proving the main results. Moreover,  $\mathcal{S}_{\omega}(\mathbb{R})$ is a nuclear Fr\'echet space (see \cite{nuclear}). Several equivalent systems of semi-norms describing the topology of $\mathcal{S}_{\omega}(\mathbb{R})$ can be found in \cite{asensio_jornet,seminormas}. For example, 
\[
q_{\omega,\lambda,\mu}(f):=\sup_{x\in {\mathbb R}}\sup_{n\in {\mathbb N}_0}|f^{(n)}(x)|\exp\left(-\lambda\varphi_\omega^\ast\left(\frac{n}{\lambda}\right)\right) \exp\mu\omega(x), 
\]
for $\lambda, \mu > 0$, 
and
\[
q_{\omega,\lambda}(f):=\sup_{x\in {\mathbb R}}\sup_{n\in {\mathbb N}_0}|f^{(n)}(x)|\exp\left(-\lambda\varphi_\omega^\ast\left(\frac{n}{\lambda}\right)\right) \exp\lambda\omega(x), 
\]
for $\lambda>0$, are  equivalent systems of semi-norms on $\mathcal{S}_{\omega}(\mathbb{R})$.\\

We also note that by \cite{mt} condition ($\varepsilon$) is equivalent to the surjectivity of the Borel map
$$
B\colon {\mathcal S}_\omega({\mathbb R})\to {\mathcal S}_\omega(\{0\}),\quad f\mapsto (f^{(j)}(0))_{j\in \mathbb{N}_0},
$$
where
$$
{\mathcal S}_\omega(\{0\}):=\left\{(x_j)_{j\in \mathbb{N}_0}\in \mathbb{C}^{\mathbb{N}_0}:\ \sup_{j\in\mathbb{N}_0}|x_j|\exp\left(-k\varphi^*_\omega\left(\frac{j}{k}\right)\right)<\infty\ \forall k\in\mathbb{N}\right\}.
$$

The Gelfand-Shilov spaces can be  also defined by mean of a weight sequence. 

\begin{definition}\label{weight-sequence} A
	sequence $M=(M_p)_{p\in\mathbb{N}_0}$ is said to be  a weight sequence if it satisfies
	\begin{itemize}
		\item[$(M_0)$] There exists $c >0$ such that $(c(p +1))^p\leq M_p$, for $p\in\mathbb{N}_0$.
		\item[$(M_1)$] $M_{2p}\leq M_{p-1}M_{p+1}$, for  $p\in\mathbb{N}$ and $M_0=1$.
		\item[$(M_2)$] There are $A, H>0$ such that $M_p\leq AH^p\min_{0\leq q\leq p}M_qM_{p-q}$, for $p \in\mathbb{N}_0$.
		\item[$(M_3)'$] $\sup_{p\in\mathbb{N}}\frac{m_p}{p}\sum_{j\geq p}\frac{1}{m_j}<\infty$, where $m_p:=\frac{M_p}{M_{p-1}}$, for $p\in\mathbb{N}$.
	\end{itemize}
\end{definition}

\begin{definition}\label{GS-sequence} Let $M=(M_p)_{p\in\mathbb{N}_0}$ be a weight sequence. The Gelfand-Shilov space of Beurling type $\mathcal{S}_M({\mathbb{R}})$ consists of those functions $f\in C^\infty(\mathbb{R})$ with the property that 
	\[
	p_{M,h}(f):=\sup_{x\in\mathbb{R}}\sup_{n,q\in\mathbb{N}_0}\frac{|x|^q|f^{(n)}(x)|}{h^{n+q}M_{n+q}}<\infty,
	\]
	for all $h>0$.
\end{definition}

Note that condition $(M_3)'$ implies condition $(M_3)$ of non quasi-analiticity, i.e.,  $\sum_{p=1}^\infty\frac{M_{p-1}}{M_p}<\infty$. Thus, $\mathcal{S}_M({\mathbb{R}})$   contains non trivial compactly supported functions. Moreover, 
$\mathcal{S}_M({\mathbb{R}})$ is a Fr\'echet space when it is equipped with the topology generated by the system $\{p_{M,h}\}_{h>0}$ of norms. \\

The most relevant example is the classical Gelfand-Shilov space of index $s>1,$ $\Sigma_s(\mathbb{R})$ defined as the space of all functions $f\in C^\infty(\mathbb{R})$ such that
\[
p_\mu(f):=\sup_{j,q\in\mathbb{N}_0}\sup_{x\in \mathbb{R}}|x|^q|f^{(j)}(x)|\frac{\mu^{j+q}}{j!^sq!^s}<\infty,
\]
for all $\mu>0$.\\

Recall that if the weight function $\omega$ satisfies the condition
\begin{equation}\label{C.Sequence}
	(\zeta)\quad \exists H\geq 1:\ 2\omega(t)\leq \omega(Ht)+H, \ \forall t\geq 0,
\end{equation}
then there exists a weight sequence $M:=\{M_p\}_{p\in\mathbb{N}_0}$ such that $\mathcal{S}_M({\mathbb{R}})=\mathcal{S}_\omega(\mathbb{R})$ algebraically and topologically, see \cite[Corollary 16]{BMM} and \cite[Section 3]{BJON}. We refer to \cite{jmaa} to find more on the connection with Gelfand-Shilov spaces defined in terms of a weight sequence instead of a weight function.\\

Note that $\omega(t) = t^{\frac{1}{s}}$ is a strong weight satisfying condition $(\zeta)$ for all $s>1$ and whose corresponding weight sequence is $M=(p!^s)_{p\in\mathbb{N}_0}$ i.e. $\Sigma_s({\mathbb R}) = {\mathcal S}_\omega({\mathbb R}).$ On the other hand, the function $\omega(t)=\max\{0,\log^p(t)\}$ is an example of a strong weight not satisfying condition $(\zeta)$ for all $p>1$, \cite[20 Example]{BMM}.

\subsection{Topologizabilty and $m$-topologizability}
The aim of this paper is to study the topologizability (and $m-$topologi\-zability) of the composition operators $C_\psi$ acting on Gelfand-Shilov classes, when $\psi$ is a polynomial.
For the sake of the reader, let us recall some definitions. Let $E$ be a locally convex Hausdorff space.  We denote by $cs(E)$ a system of continuous seminorms determining the topology of $E$. Given locally convex Hausdorff spaces $E$, $F$ we denote by $L(E,F)$ the space of all continuous linear operators from $E$ to $F$. If $E=F$, then we simply write $L(E)$ for $L(E, E)$.\\

The concept of topologizability and $m-$topologizability was defined originally and studied by Zelazko in  \cite{z}.
\begin{definition}
	An operator $T \in L(E)$ on a locally convex Hausdorff space $E$ is called
	topologizable if for  $p\in cs(E)$ there is  $q \in cs(E)$ such that for every $k\in\mathbb{N}$ there is $M_k>0$ such that
	\begin{equation*}
		p(T^k(x)) \leq M_k \hspace{0.03cm} q(x)
	\end{equation*} for each $x \in E$.
\end{definition}

\begin{definition}
	An operator $T \in L(E)$ on a locally convex Hausdorff  space $E$ is called $m-$topologizable if for  $p\in cs(E)$ there are  $q \in cs(E)$ and $C\geq 1$ such that 
	\begin{equation*}
		p(T^k(x)) \leq C^k \hspace{0.03cm} q(x)
	\end{equation*}
	for each $k\in\mathbb{N}$ and for each $x\in E$.
\end{definition}
Observe that in the definitions above it is essential that the seminorm $q$ only depends on the seminorm $p$ and not on the iteration $k$. Clearly, we have that power boundedness of an operator $T\in L(E)$ on a locally convex  Hausdorff space $E$ (i.e., equicontinuity of the sequence $(T^k)_{k\in{\mathbb N}}$) implies $m-$topologizability, which in turn implies topologizability.\\

As it was done with the concept of power boundedness and mean ergodicity in \cite{RACSAM}, we extend the concepts of  topologizability and of $m-$topolo\-gizability to an arbitrary family of operators $\{T_m:E\to F: m\in\mathbb{N}\},$ where $E$ and $F$ are locally convex Hausdorff spaces: 

\begin{definition} Let $E$, $F$ be locally convex Hausdorff spaces.
	A family of operators $\{T_m:E\to F: m\in\mathbb{N}\}$ is called
	topologizable if for every $p\in cs(F)$ there is  $q \in cs(E)$ such that for every $k\in\mathbb{N}$ there is $M_k>0$ such that \begin{equation*}
		p(T_k(x)) \leq M_k \hspace{0.03cm} q(x)
	\end{equation*} for each $x \in E$.
\end{definition}

\begin{definition} Let $E$, $F$ be locally convex Hausdorff spaces.
	A family of operators $\{T_m:E\to F: m\in\mathbb{N}\}$ is called $m-$topologizable if for every  $p\in cs(F)$ there are $q \in cs(E)$ and $C\geq 1$ such that
	\begin{equation*}
		p(T_k(x)) \leq C^k \hspace{0.03cm} q(x)
	\end{equation*} for each $k\in\mathbb{N}$ and for each $x\in E$. 
\end{definition}

In the above definitions, it is essential yet again that the  semi-norm $q$ only depends on the semi-norm $p$ and not on the iteration $k$; otherwise, all families of continuous linear operators $\{T_m:E\to F: m\in\mathbb{N}\}$ would fulfill them automatically. 
Also in this case, equicontinuity of the family of operators $\{T_m:E\to F: m\in\mathbb{N}\}$ implies $m-$topologizability,
which in turn implies topologizability.

\section{Topologizability (and $m-$topologizability) of the composition operators associated to a polynomial $\psi$ of degree greater than one on Gelfand-Shilov classes.} \label{sec2}

In this section we prove that even when the polynomial $\psi$ has fixed points we still have that the family of iterates $\{C_{\psi_m}:\mathcal{S}_{\omega}(\mathbb{R})\to \mathcal{S}_{\sigma}(\mathbb{R}):m\in\mathbb{N}\}$ is topologizable in the same setting, i.e. where $\sigma=\omega(\bullet^{\frac{1}{a}})$, with $a>2$:

\begin{theorem}\label{toppolynomialsfirst}
	Let $\omega$ be a sub-additive weight function and $\psi$ be a polynomial of degree strictly greater than one. Then the family of iterates $\{C_{\psi_m}:\mathcal{S}_{\omega}(\mathbb{R})\to \mathcal{S}_{\sigma}(\mathbb{R}):m\in\mathbb{N}\}$ is topologizable, whenever $\sigma=\omega(\bullet^{\frac{1}{a}})$ and $a>2$.
\end{theorem}
\begin{proof}
	We recall that $\varphi^*_\sigma(x)=\varphi^*_\omega(ax)$ for all $x\geq 0$. Let $\rho_m$ be the degree of the polynomial $\psi_m$ and $I_j=\{(k_1,...,k_j)\in\mathbb{N}_0^j: k_1+2k_2+\ldots+j k_j=j\}$, for every $m,j\in\mathbb{N}$. The fact that $\psi$ is a polynomial of degree greater than one clearly implies that there are $\alpha\in ]1,2[$ and $b>1$ such that $|\psi(x)|\geq |x|^\alpha$ for all $|x|\geq b$. So, it easily follows that there is $C_0\geq 1$ such that $1+|x|\leq C_0 (1+|\psi_m(x)|)$ for all $x\in\mathbb{R}$ and $m\in\mathbb{N}$. Since $\psi_m^{(j)}=0$ for every $m\in\mathbb{N}$ and $j\in\mathbb{N}$ with $j>\rho_m$,
	we also have that for each $m\in\mathbb{N}$ there is $D_m>0$ such that \begin{equation*}
		|\psi_m^{(\ell)}(x)|\leq D_m \hspace{0.05cm} (1+|\psi_m(x)|)^{\delta_m}
	\end{equation*} 
	for all $x\in\mathbb{R}$, $\ell\in\mathbb{N}$, where $0<\delta_m=\frac{\rho_m-1}{\rho_m}<1$. Therefore, for every $j,m\in \mathbb{N}$ and $(k_1,\ldots,k_j)\in I_j$ we obtain that
	\begin{equation*}
		\prod_{\ell=1}^j \left|\frac{\psi_m^{(\ell)}(x)}{\ell!}\right|^{k_\ell}\leq D_m^k(1+|\psi_m(x)|)^{\delta_mk}\prod_{\ell=1}^j\frac{1}{(\ell!)^{k_\ell}}
	\end{equation*}
	for all $x\in \mathbb{R}$, where $k=k_1+k_2+\ldots k_j$, and hence,  for every $\mu>0$ and  $x\in \mathbb{R}$ that 
	\begin{equation*}
		\prod_{\ell=1}^j \left|\frac{\psi_m^{(\ell)}(x)}{\ell!}\right|^{k_\ell}\leq D_m^k(1+|\psi_m(x)|)^{k}\prod_{\ell=1}^j\frac{\exp\left(2k_\ell\mu \varphi^*_\omega(\frac{\ell}{2\mu})\right)}{(\ell!)^{k_\ell}}\cdot \frac{1}{\exp\left(2k_l\mu \varphi^*_\omega(\frac{\ell}{2\mu})\right)}.
	\end{equation*}
	Since $\omega$ is a sub-additive weight, we can now argue as in the proof of \cite[Proposition 2.1]{fg} to show 
	that for every $\mu>0$ and $m\in\mathbb{N}$ there is $B_{m,\mu}\geq 1$ such that 
	\begin{align*}
		\prod_{\ell=1}^j \left|\frac{\psi_m^{(\ell)}(x)}{\ell!}\right|^{k_\ell}&\leq B_{m,\mu}^k \frac{\exp\left(\mu \varphi^*_\omega(\frac{j-k}{\mu})\right)}{(j-k)!} (1+|\psi_m(x)|)^k \prod_{\ell=1}^j \frac{1}{\exp\left(2k_\ell\mu \varphi^*_\omega(\frac{\ell}{2k_\ell\mu})\right)}\\
		&\leq B_{m,\mu}^k \frac{\exp\left(\mu \varphi^*_\omega(\frac{j-k}{\mu})\right)}{(j-k)!} (1+|\psi_m(x)|)^k 
	\end{align*}
	for all $(k_1,...,k_j)\in I_j$, $j\in \mathbb{N}$ and $x\in\mathbb{R}$, where $k=k_1+...+k_j$. \\
	
	Fix $\lambda>0$ and $\delta>0$. Then there exist   $L_\lambda>0$ and  $\mu= \mu(\lambda)>0$ such that 
	\begin{equation*}
		4^j\hspace{0.05cm}  C_0^q\hspace{0.05cm}  \exp\left(\mu \varphi^*_\omega\left(\frac{(2+\delta)(q+j)}{\mu}\right)\right) \leq L_{\lambda} \exp\left(\lambda \varphi^*_\omega\left(\frac{(2+\delta)(q+j)}{\lambda}\right)\right)
	\end{equation*}
	for every $j,q\in \mathbb{N}$ (cf. \cite[Lemma 3.3]{jmaa} or \cite[Lemma A.1]{paley}). On the other hand, 
	for each $m\in\mathbb{N}$ there is $\square_{m,\mu,\delta}>0$ such that \begin{equation*}
		B_{m,\mu}^k\leq \square_{m,\mu,\delta} \exp\left(\mu \varphi^*_\omega\left(\frac{\delta k}{\mu}\right)\right)
	\end{equation*} 
	for all $k\in\mathbb{N}$ (cf. \cite[Lemma 3.3]{jmaa} or \cite[Lemma A.1]{paley}). \\
	
	By the  Fa\'a Di Bruno's formula and the fact that $$\sum_{(k_1,...,k_j)\in I_j} \frac{(k_1+\ldots+k_j)!}{k_1! \ldots k_j!}=2^{j-1}$$ for all $j\in\mathbb{N}$ (see, for instance, \cite{jmaa}  for a  proof of this identity), it follows that 
	\begin{align*}
		&|(1+|x|)^q (f\circ \psi_m)^{(j)}(x)|\\
		&\leq  
		C_0^q (1+|\psi_m(x)|)^q \sum_{(k_1,...,k_j)\in I_j} \frac{j!}{k_1! \ldots k_j!} |f^{(k)}(\psi_m(x))| \prod_{\ell=1}^j \left|\frac{\psi_m^{(\ell)}(x)}{\ell!}\right|^{k_\ell}\\
		&\leq  
		C_0^q \sum_{(k_1,...,k_j)\in I_j} \frac{j!}{k_1! \ldots k_j!} |f^{(k)}(\psi_m(x))| B_{m,\mu}^k \frac{\exp\left(\mu \varphi^*_\omega\left(\frac{j-k}{\mu}\right)\right)}{(j-k)!} (1+|\psi_m(x)|)^{k+q}   \\
		& \leq 2^j  C_0^q   p_{\omega,\mu}(f) \hspace{0.02cm} \sum_{(k_1,...,k_j)\in I_j} \frac{k!}{k_1! \ldots k_j!}  \exp\left(\mu \varphi^*_\omega\left(\frac{2k+q+j-k}{\mu}\right)\right) B_{m,\mu}^k \\
		& \leq \square_{m,\mu,\delta} 2^j  C_0^q p_{\omega,\mu}(f)    \sum_{(k_1,...,k_j)\in I_j} \frac{k!}{k_1! \ldots k_j!}  \exp\left(\mu \varphi^*_\omega\left(\frac{(1+\delta)k+q+j}{\mu}\right)\right)     \\ 
		&\leq \square_{m,\mu,\delta} \hspace{0.05cm} 4^j  C_0^q p_{\omega,\mu}(f)   \exp\left(\mu \varphi^*_\omega\left(\frac{(2+\delta)(q+j)}{\mu}\right)\right) \\ &\leq \square_{m,\mu,\delta} L_{\lambda} \exp\left(\lambda \varphi^*_\omega\left(\frac{(2+\delta)(q+j)}{\lambda}\right)\right) p_{\omega,\mu}(f)
\end{align*}
for all $j\in\mathbb{N}, q\in\mathbb{N}_0$, $x\in\mathbb{R}$, $m\in\mathbb{N}$ and $f\in \mathcal{S}_{\omega}(\mathbb{R})$. Since $\lambda>0$ and $\delta>0$ were arbitrary and $\varphi^*_\sigma(s)=\varphi^*_\omega(as)$ with $a>2$, for $s\geq 0$,  the proof is complete.
\end{proof}

As in \cite[Corollary 3.5]{RACSAM}, we can deduce the following immediate consequence in the case where the weight $\omega$ satisfies the additional rather technical condition \eqref{logcond}, which power of logarithms are particular cases of: 
\begin{corollary}
Let $\omega$ be a sub-additive weight such that the following condition is satisfied: 
\begin{equation}\label{logcond}    
	\exists \gamma > 1\ \exists C\geq 1\ \forall t\geq 0:\ \omega(t^\gamma)\leq C\omega(t) + C. 
\end{equation} 
If $\psi$ is a polynomial of degree strictly greater than one, then the family of iterates $\{C_{\psi_m}:\mathcal{S}_{\omega}(\mathbb{R})\to \mathcal{S}_{\omega}(\mathbb{R}):m\in\mathbb{N}\}$ is topologizable. 
\end{corollary}

In the following result we show that Theorem \ref{toppolynomialsfirst} cannot be improved to obtain the $m$-topologizability of the composition operator $C_\psi$ in the same setting, unlike what happens in the classical Schwartz space setting (see, for instance, \cite[Example 4.14.]{akj}). In particular, we establish that there is at least one polynomial $\psi$ of degree greater than one and a sub-additive weight $\omega$ for which $C_\psi: \mathcal{S}_{\omega}(\mathbb{R})\to \mathcal{S}_{\omega(\bullet^{\frac{1}{2}})}(\mathbb{R})$ is not $m-$topologizable.

\begin{prop}\label{squaremtop}
If $\psi(x)=x^2$ for all $x\in\mathbb{R}$ and $\omega(t)=|t|^{\frac{1}{s}}$, with $s>1$, then $\{C_{\psi_m}:\mathcal{S}_{\omega}(\mathbb{R})\to \mathcal{S}_{\omega(\bullet^{\frac{1}{2}})}(\mathbb{R}):m\in\mathbb{N}\}$ is not $m-$topologizable.
\end{prop}

\begin{proof}
We denote $\sigma:=\omega(\bullet^{\frac{1}{2}})$. We observe that 
\begin{equation*}
	\exp\left(-\lambda \varphi^*_\sigma\left(\frac{m}{\lambda}\right)\right)=\left(\frac{\lambda e}{2s m}\right)^{2s m}
\end{equation*} 
for all $m\in\mathbb{N}$ and $\lambda>0$, as  an easy computation shows.
We also  note that $\psi_m(x)=x^{2^m}$, for all $x\in\mathbb{R}$ and $m\in\mathbb{N}$.\\

We  suppose that the family of iterates $\{C_{\psi_m}:\mathcal{S}_{\omega}(\mathbb{R})\to \mathcal{S}_{\sigma}(\mathbb{R}):m\in\mathbb{N}\}$ is $m-$topologizable, i.e. for every $\lambda>0$ there are $\mu>0$ and $C>0$ such that $p_{\sigma,\lambda}(C_{\psi_m}f)\leq C^mp_{\omega,\mu}(f)$ for all $m\in\mathbb{N}$ and  $f\in \mathcal{S}_{\omega}(\mathbb{R})$. Accordingly, 
\begin{equation}\label{redabssquare}
	\sup_{x\in\mathbb{R}, j, q\in\mathbb{N}_0} (1+|x|)^q  |(f\circ \psi_m)^{(j)}(x)|  \exp\left(-\lambda \varphi^*_\sigma\left(\frac{j+q}{\lambda}\right)\right) \leq C^m  p_{\omega,\mu}(f)
\end{equation} 
for all $m\in\mathbb{N}$ and  $f\in \mathcal{S}_{\omega}(\mathbb{R})$. Since $\omega$ is a strong weight, we can select $f\in \mathcal{S}_{\omega}(\mathbb{R})$ so that $f'(1)=1$ and $f^{(h)}(1)=0$ for all $h\geq 2$. 
By Fa\'a Di Bruno's formula, it follows that 
\begin{equation*}
	|(f\circ\psi_m)^{(j)}(1)|=2^m (2^m-1) \ldots (2^m-j+1)
\end{equation*} 
for all $m\in\mathbb{N}$, $j\leq 2^m$. Observe that  the following inequality is satisfied 
\begin{equation*}
	2^m (2^m-1) \ldots (2^m-m+1)\geq (2^m-m+1)^m \geq 2^{\frac{m^2}{2}}
\end{equation*} 
for all $m\in\mathbb{N}$ large enough.

If we put $q=0,x=1,j=m$ in \eqref{redabssquare} and use the  inequality above, we get the following estimate: \begin{equation*}
	p_{\omega,\mu}(f)\hspace{0.03cm} C^m \geq  [2^m(2^m-1) \ldots (2^m-m+1)]\exp\left(-\lambda \varphi^*_\sigma\left(\frac{m}{\lambda}\right)\right) \geq 2^{\frac{m^2}{2}} \left(\frac{\lambda e}{2s m}\right)^{2s m}
\end{equation*}
for all $m\in\mathbb{N}$ large enough.  This would imply that there is $Q=Q(f,\lambda, s)>0$ such that 
\begin{equation*}
	Q\geq \frac{2^{\frac{m}{2}}}{m^{2s}}
\end{equation*}
for all $m\in\mathbb{N}$, which is a contradiction with the obvious fact that
\begin{equation*}
	\frac{2^{\frac{m}{2}}}{m^{2s}}\to \infty \mbox{ as } m\to \infty.\qedhere
\end{equation*} 
\end{proof}

Notice that $\psi'(1)=2>1$ and $\psi(1)=1$, i.e. $1$ is a repelling fixed point of $\psi$. Let us show the following rather more general result: \begin{theorem}\label{repellingfixedpoint}
Let $\psi$ be a polynomial of degree greater than one that possesses at least one repelling fixed point, i.e. there is $x_0\in\mathbb{R}$ such that $\psi(x_0)=x_0$ and $|\psi'(x_0)|>1$. Then, the family of iterates $\{C_{\psi_m}:\mathcal{S}_{\omega}(\mathbb{R})\to \mathcal{S}_{\omega(\bullet^{\frac{1}{2}})}(\mathbb{R}):m\in\mathbb{N}\}$ is not $m-$topologizable, where $\omega=|\bullet|^{\frac{1}{d}}$, with $d>1$.
\end{theorem}

\begin{proof}
Set $\sigma:= \omega(\bullet^{\frac{1}{2}})$ and $\alpha:=|\psi'(x_0)|>1$. Proceeding by contradiction, we assume that for every $\lambda>0$ there are $\mu>0$, $C>0$ such that \begin{equation}\label{contradictionmtop}
	p_{\sigma, \lambda}(f\circ \psi_m)\leq C^m \hspace{0.05cm} p_{\omega, \mu}(f)
\end{equation} for all $f\in \mathcal{S}_{\omega}(\mathbb{R})$, $m\in\mathbb{N}$. Fix $\lambda>0$ and $\mu>0$ so that \eqref{contradictionmtop} holds. We also recall that 
\begin{equation*}
	\varphi_\omega^\ast(x) = xd\log\left(\frac{xd}{e}\right), \varphi_\sigma^\ast(x)=2xd\log\left(\frac{2xd}{e}\right)\end{equation*} 
for all $x\geq 0$, and hence, we have that $$\exp(-\lambda \varphi_\sigma^\ast(\frac{j}{\lambda}))=\left(\frac{\lambda e}{2dj}\right)^{2dj}=A_{\lambda, d}^j\hspace{0.05cm}\frac{1}{j^{2dj}},$$ for all $j\in\mathbb{N}$, where $A_{\lambda, d}=(\frac{\lambda e}{2d})^{2d}>0$. Without loss of generality, we may assume $A_{\lambda, d}>1$; otherwise we take $\lambda>0$ larger in order to ensure so. \\

Proceeding as in the proof of \cite[Theorem 3.5.]{jmaa}, since $$a_j=\exp\left(\log(j)\varphi_\omega^\ast(\frac{j}{\log(j)})\right)=B^j\hspace{0.05cm} \left(\frac{j}{\log(j)}\right)^{jd},$$ for all $j\in\mathbb{N}$, where $B\equiv B_d=(\frac{d}{e})^d>0$, we can construct a sequence of functions $(f_\ell)_\ell \subset \mathcal{S}_{\omega}(\mathbb{R})$ that is bounded in $\mathcal{S}_{\omega}(\mathbb{R})$ and that verifies: \begin{equation*}
	f_\ell^{(j)}(x_0)=\delta_{\ell}^j \hspace{0.05cm} B^j \hspace{0.05cm}\left(\frac{j}{\log(j)}\right)^{jd},
\end{equation*} for all $\ell\in\mathbb{N}, j\in\mathbb{N}$, where $\delta_{\ell}^j=0$ whenever $\ell\not=j$ and $\delta_{\ell}^j=1$ otherwise. \\

Obviously, $\psi_m(x_0)=x_0$ and $\psi_m'(x_0)=\psi'(x_0)^m$ for all $m\in\mathbb{N}$. So, after applying Faa Di Bruno's formula (see, for instance, \cite[Lemma 4.1.]{jmaa}), we obtain that \begin{equation*}\begin{split}
		|(f_j\circ \psi_m)^{(j)}(x_0)|\hspace{0.05cm} A_{\lambda, d}^j\hspace{0.05cm}\frac{1}{j^{2dj}}&=  |f_j^{(j)}(x_0)\psi_m'(x_0)^j| \hspace{0.05cm} A_{\lambda, d}^j\hspace{0.05cm}\frac{1}{j^{2dj}}\\&=(A_{\lambda, d}\hspace{0.05cm}B)^j\hspace{0.05cm}\left(\frac{j}{\log(j)}\right)^{jd} \hspace{0.05cm} \alpha^{jm} \hspace{0.05cm}\frac{1}{j^{2dj}} \\&=(A_{\lambda, d}\hspace{0.05cm}B)^j\hspace{0.05cm}\left(\frac{1}{j\log(j)}\right)^{jd} \hspace{0.05cm} \alpha^{jm}
	\end{split}
\end{equation*} for all $j\in\mathbb{N}$, $m\in\mathbb{N}$. Now, putting $j=m$ in the above equality, we get \begin{equation}\label{loginequalitymtop}
	|(f_m\circ \psi_m)^{(m)}(x_0)|\hspace{0.05cm} A_{\lambda, d}^m\hspace{0.05cm}\frac{1}{m^{2dm}}= \left(\frac{1}{m\log(m)}\right)^{md} \hspace{0.05cm} \alpha^{m^2} \hspace{0.05cm} (A_{\lambda, d}\hspace{0.05cm} B)^m
\end{equation} for all $m\in\mathbb{N}$. Using the definition of $p_{\sigma, \lambda}(f\circ \psi_m)$ (with $q=0, x=x_0, j=m$), the fact that there is $D_\mu>0$ so that $p_{\omega, \mu}(f_m)\leq D_\mu,$ for all $m\in\mathbb{N}$ and combining \eqref{contradictionmtop} with \eqref{loginequalitymtop}, we obtain: \begin{equation*}
	D_\mu C^m \geq \left(\frac{1}{m\log(m)}\right)^{md} \hspace{0.05cm} \alpha^{m^2} \hspace{0.05cm} (A_{\lambda, d} \hspace{0.05cm} B)^m
\end{equation*} for all $m\in\mathbb{N}$. In turn, this implies that there is a constant $H_{\lambda,\mu, d}>0$ such that \begin{equation*}
	\left(\frac{\alpha^{\frac{m}{d}}}{m\log(m)}\right)^{d} \leq H_{\lambda,\mu, d}
\end{equation*} for all $m\in\mathbb{N}$, which is a contradiction with the fact that $$\left(\frac{\alpha^{\frac{m}{d}}}{m\log(m)}\right)\to \infty \mbox{ as } m\to \infty.\qedhere
$$
\end{proof}

Recall that if $\psi$ is a polynomial of degree greater than or equal to two, then the equicontinuity of the composition operator $C_\psi:\mathcal{S}_{\omega}(\mathbb{R})\to \mathcal{S}_{\sigma}(\mathbb{R})$, with $\sigma=\omega(\bullet^{\frac{1}{a}})$ for $a>2$, is equivalent to the lack of fixed points of the polynomial $\psi$ (see \cite{RACSAM}). On the other hand, equicontinuity implies always $m$-topologizability. Let us observe that if $\psi(x)=x^2+\frac{1}{4}$ one has that $\psi(\frac{1}{2})=\frac{1}{2}$ and $\psi'(\frac{1}{2})=1$. So we could not apply Theorem \ref{repellingfixedpoint} and hence, we do not know yet if the associated composition operator $C_\psi:\mathcal{S}_{\omega}(\mathbb{R})\to \mathcal{S}_{\omega(\bullet^{\frac{1}{2}})}(\mathbb{R})$ is $m-$topologizable or not. More generally, the following problem arises: 

\begin{remark}\rm It is an open problem to establish
what happens with Theorem \ref{repellingfixedpoint} if the polynomial $\psi$ appearing therein only have non-repelling fixed points. An example of such polynomial is $\psi(x)=x^2+\frac{1}{4}$.  
\end{remark}

\section{Topologizability (and m-topologizability) of the composition operator associated with a polynomial $\psi$ of degree one on Gelfand-Shilov classes}\label{sec3}

The aim of this section is to study the topologizability of the composition operator $C_\psi$ acting on Gelfand-Shilov classes, when the polynomial $\psi$ is of degree one. In this case, we know that the composition operator $C_\psi: \mathcal{S}_{\omega}(\mathbb{R})\to \mathcal{S}_{\omega}(\mathbb{R})$ is continuous and hence, this is the usual and proper setting to work in. In \cite[Proposition 3.1]{RACSAM}, it was proved that if $\psi(x)=ax+b$, with $a\not=0$, for $x\in\mathbb{R}$ then, $C_\psi: \mathcal{S}_{\omega}(\mathbb{R})\to \mathcal{S}_{\omega}(\mathbb{R})$ is power bounded if and only if $C_\psi: \mathcal{S}_{\omega}(\mathbb{R})\to \mathcal{S}_{\omega}(\mathbb{R})$ is mean ergodic if and only if $\psi(x)=x$ for $x\in\mathbb{R}$ or $\psi(x)=-x$ for $x\in\mathbb{R}$. \\

Recall that two polynomials $\psi, \phi$ are said to be linearly equivalent if there exists $\ell(x)=\alpha x+\beta$, for $x\in\mathbb{R}$, with $\alpha, \beta\in\mathbb{R}$ and $\alpha\not = 0,$ such that $\phi(x)=(\ell \circ \psi\circ\ell^{-1})(x)$ for all $x\in\mathbb{R}$. For what follows it is convenient to observe  that the following result is valid, as it is easy to  show.

\begin{prop} Let $\psi(x) = a x + b$, for  $x\in\mathbb{R}$, with $a,b\in {\mathbb R}$ and $a\neq 0.$.
\begin{itemize}
	\item[(a)] If $a\neq 1,$ then $\psi$ is linearly equivalent to $\phi(x) = ax$, for $x\in \mathbb{R}$. 
	\item[(b)] If $a = 1$ and $b\neq 0,$ then $\psi$ is linearly equivalent to $\phi(x) = x + 1,$ for $x\in\mathbb{R}$.
\end{itemize}
\end{prop}
It is useful to notice that if two polynomials $\psi, \phi$ are linearly equivalent, then the composition operator $C_\psi$ is clearly topologizable ($m-$topologizable, respectively) if and only if the composition operator $C_\phi$ is topologizable ($m-$topologizable, respectively). The proof is almost immediate when one observes that if $\psi=\ell \circ \phi\circ \ell^{-1}$, with $\ell$ being a non-constant affine function, then $C_{\psi_m}=C_\ell \circ C_{\phi_m}\circ C_{\ell}^{-1}$ for all $m\in\mathbb{N}$, where  the composition operator $C_\ell: \mathcal{S}_{\omega}(\mathbb{R})\to \mathcal{S}_{\omega}(\mathbb{R})$ is clearly an isomorphism onto. Therefore,  we may assume without loss of generality that either $\psi(x)=a x$, for  all $x\in\mathbb{R}$, with $a\not =0, \pm 1$,  or $\psi(x)=x+1$, for all $x\in\mathbb{R}$.\\

			First, we  deal with the case where $\psi$ is a translation. We know that composition operators acting on Gelfand-Shilov classes associated with translations are not power bounded (see \cite[Proposition 3.1]{RACSAM}) but, just as in the Schwartz class (see \cite[Example 4.14.]{akj}), we still have $m-$topologizability.

			\begin{prop}\label{toptrans}
				Let $\omega$ be a weight function and $\psi(x)=x+1$, for  $x\in\mathbb{R}$. Then the composition operator $C_\psi: \mathcal{S}_{\omega}(\mathbb{R})\to \mathcal{S}_{\omega}(\mathbb{R})$ is m-topologizable.
			\end{prop}
			\begin{proof} Fix $\lambda > 0, \mu > 0$ and observe that  that $$q_{\omega,\lambda, \mu}(C_{\psi_m}f)
				=\sup_{j\in {\mathbb N}_0}\sup_{y\in {\mathbb R}} |f^{(j)}(y)| e^{-\lambda\varphi_\omega^\ast(\frac{j}{\lambda}) + \mu\omega(y-m)},
				$$
				for every $f\in \mathcal{S}_{\omega}(\mathbb{R})$ and $m\in\mathbb{N}$.\\
				
				Since $\mu\hspace{0.05cm} \omega(y-m)\leq \mu\hspace{0.05cm} L\left(1+\omega(y)+\omega(m)\right)$, for all $y\in\mathbb{R}$ and $m\in\mathbb{N}$, it follows, for every $f\in\mathcal{S}_{\omega}(\mathbb{R})$, $\lambda>0$ and $\mu>0$, that 
				$$
				q_{\omega, \lambda, \mu}(C_{\psi_m}f)\leq \exp\left(\mu\hspace{0.05cm} L (1+\omega(m))\right)q_{\omega,\lambda,\mu\hspace{0.05cm} L}(f),
				$$
				for all $m\in\mathbb{N}$.\\
				
				We now observe that condition $(\beta)$ in Definition \ref{def:weight} together with the fact that $\omega$ is increasing imply  that $\frac{\omega(t)}{t}\to 0$ as $t\to \infty$ (cf. \cite[1.2.Remark(b)]{mt}). 
					Therefore, there is $Q>0$ such that $\omega(t)\leq Q t$ for all $t\geq 0$. It follows for every $f\in\mathcal{S}_{\omega}(\mathbb{R})$ that
					$$
					q_{\omega,\lambda, \mu}(C_{\psi_m}f)\leq \exp\left(\mu\hspace{0.05cm} L\right)\exp\left(\mu\hspace{0.05cm} L Q \hspace{0.05cm} m\right)q_{\omega,\lambda,\mu\hspace{0.05cm} L}(f),$$ 
					for all $m\in\mathbb{N}$. This clearly completes the proof.
				\end{proof}
				
				\begin{remark}
					This result is still true for weaker classes of Braun-Meise-Taylor weight functions $\omega$, provided that $\omega$ verifies the condition that $\omega(t)\leq Q t$ for  $t>0$ large enough and for some $Q>0$ (see, for instance, \cite{Schindl} and the references therein for a survey of such classes). 
				\end{remark}

				In the setting of the Gelfand-Shilov spaces $\Sigma_s(\mathbb{R})$, $s>1$,   an interesting and non-obvious consequence follows from  Proposition \ref{toptrans} due to the  use of another equivalent system of semi-norms for $\Sigma_s(\mathbb{R})$, as given in Section \ref{sec1.1}.

				\begin{corollary}
					Let $s>1$. Then there is $\lambda>0$ such that for every $A>0$ there is no $f\in \Sigma_s(\mathbb{R})\setminus\{0\}$   such that the conditions  $supp f\subset [-A,A]$ and  $p_{\lambda}(f)$ is attained in $j,q\in\mathbb{N}_0$ with $q-j$  arbitrarily large are simultaneously satisfied.  
				\end{corollary}
				\begin{proof}
					If $\psi(x)=x+1$, for $x\in\mathbb{R}$, by Proposition \ref{toptrans} the composition operator $C_\psi\colon \Sigma_s(\mathbb{R})\to \Sigma_s(\mathbb{R})$ is $m$-topologizable. Therefore,  there  exist $\mu\geq 1$ and $D\geq 1$  such that for each $m\in \mathbb{N}$ 
					\begin{equation}\label{contradiction1}
						p_{1}(f(\bullet+m))= \sup_{j,q\in\mathbb{N}_0}\sup_{x\in\mathbb{R}} \left(\frac{|x|}{|x+m|}\right)^q\frac{|x+m|^q \hspace{0.05cm}|f^{(j)}(x+m)|}{j!^s q!^s}
						\leq D^m p_{\mu}(f)
					\end{equation} 
					for all $f\in \Sigma_s(\mathbb{R})$. Without loss of generality, we may assume that $\log_{\mu}(D)\geq 1$ and hence,  $\log_{\mu}(D^m)\geq 1$ for all $m\in\mathbb{N}$.
					
					Proceeding by contradiction, we suppose that there exist $A>0$ and sequences $(f_m)_m \subset  \Sigma_s(\mathbb{R})\setminus\{0\}$ with $supp f_m\subset [-A,A]$, $(q_m)_m\subset \mathbb{N}_0$ and $(j_m)_m\subset \mathbb{N}_0$ such that $p_\mu(f_m)= \sup_{x\in\mathbb{R}}|x|^{q_m} |f_m^{(j_m)}(x)| \frac{\mu^{j_m+q_m}}{j_m!^s q_m!^s}$ and $q_m-j_m\geq \log_\mu(D^m)>0$, for all $m\in\mathbb{N}$. Since $\frac{m-A}{A}\to\infty$ as $m\to\infty$, we have  that $\frac{m-A}{A}>\mu^2$ for all $m\in\mathbb{N}$ large enough. So, by inequality \eqref{contradiction1} we obtain that \begin{equation*}
							D^m p_\mu(f_m) \geq \left(\frac{m-A}{A}\right)^{q_m} \left(\frac{1}{\mu}\right)^{j_m+q_m} p_{\mu}(f_m) > \mu^{ q_m-j_m}p_{\mu}(f_m)> D^m p_{\mu}(f_m)
					\end{equation*} 
					for all $m\in\mathbb{N}$ large enough. This  is clearly a contradiction.
				\end{proof}
				
				We now  investigate the topologizability of $C_\psi$ when $\psi$ is a dilatation, i.e., $\psi(x)=ax$, for $x\in \mathbb{R}$, with $a\not=0$. To this end, some results are needed.

				\begin{lemma}\label{lemmaBeurling} Let $\omega$ be a weight function, 
					$f\in \mathcal{S}_\omega(\mathbb{R})$ and  $\lambda>0$. Then  for every $\varepsilon>0$ there is $M= M(\varepsilon)>0$ such that for all $j, q\in\mathbb{N}_0$ verifying $j+q\geq M$ one has that $\sup_{x\in\mathbb{R}} |x|^q |f^{(j)}(x)| \exp\left(-\lambda \varphi_\omega^*(\frac{j+q}{\lambda})\right)\leq \varepsilon$.
				\end{lemma}
				\begin{proof}
					For the given $\lambda>0$  there are $\mu>0$, $A>1$ and $D>0$ such that 
					\begin{equation*}
						\exp\left(-\lambda \varphi_\omega^*\left(\frac{j+q}{\lambda}\right)\right)\leq D \hspace{0.01cm} \left(\frac{1}{A}\right)^{j+q}\hspace{0.03cm} \exp\left(-\mu\varphi_\omega^*\left(\frac{j+q}{\mu}\right)\right)
					\end{equation*}
					for all $j,q\in\mathbb{N}_0$ (cf. \cite[Lemma 3.3]{jmaa} or \cite[Lemma A.1]{paley}). Since $f\in\mathcal{S}_\omega(\mathbb{R})$, we have  that 
					\begin{equation*}
						\sup_{x\in\mathbb{R}} |x|^q |f^{(j)}(x)| \exp\left(-\lambda \varphi_\omega^*\left(\frac{j+q}{\lambda}\right)\right)  \leq D \hspace{0.01cm} \left(\frac{1}{A}\right)^{j+q} p_{\omega,\mu}(f)
					\end{equation*} 
					for all $j,q\in\mathbb{N}_0$, with $p_{\omega,\mu}(f)<\infty$. The result easily follows. 
				\end{proof}
				
				\begin{remark}\label{R-Max} Let $\omega$ be a weight function. For $\lambda>0$ and $f\in \mathcal{S}_\omega(\mathbb{R})$ given,  
					note that Lemma \ref{lemmaBeurling} implies that there is a finite number of pairs $(j,q)\in\mathbb{N}^2_0$ for which the supremum in the norm 
					\begin{equation*}
						p_{\lambda}(f):=\sup_{j,q\in\mathbb{N}_0}\sup_{x\in\mathbb{R}} |x|^q |f^{(j)}(x)|  \exp\left(-\lambda \varphi_\omega^*\left(\frac{j+q}{\lambda}\right)\right)
					\end{equation*} 
					is attained. Note that the system $\{p_\lambda: \lambda>0\}$ of norms also generates the topology of $\mathcal{S}_\omega(\mathbb{R})$.
				\end{remark}

				\begin{prop}\label{main}
					For every $\lambda>0$ and $m\in\mathbb{N}$, there is $g\in \mathcal{S}_\omega(\mathbb{R})$ for which the supremum in $p_\lambda(g)$ is only attained in $j,q\in\mathbb{N}_0$ verifying $j-q\geq m$. In particular, $j\geq m$. 
				\end{prop}
				\begin{proof}
					Fix $f\in \mathcal{S}_\omega(\mathbb{R})\setminus \{0\}$, $\lambda>0$, $m\in\mathbb{N}$. For the sake of brevity, denote $a_{j,q}=\max_{x\in\mathbb{R}} |x|^q \hspace{0.05cm}|f^{(j)}(x)| \exp\left(-\lambda \varphi_\omega^*(\frac{j+q}{\lambda})\right),$ for all $j, q\in\mathbb{N}_0$. Observe that $a_{j,q}>0$ for all $j,q\in\mathbb{N}_0$ because $p_\lambda$ is a continuous norm over $\mathcal{S}_\omega(\mathbb{R})$. Consider the function $g(x)=f(\rho x)$, for all $x\in\mathbb{R}$, where $\rho> 1$ to be chosen later on. Clearly, $g\in \mathcal{S}_\omega(\mathbb{R})\setminus \{0\}$ and also, \begin{equation}
						\max_{x\in\mathbb{R}} |x|^q \hspace{0.05cm}|g^{(j)}(x)| \exp\left(-\lambda \varphi_\omega^*(\frac{j+q}{\lambda})\right)=\rho^{j-q} a_{j,q}
					\end{equation} for all $j,q\in\mathbb{N}_0$. Fix $m_0>m$. By Lemma \ref{lemmaBeurling}, there is $M>0$ such that for all $j,q\in\mathbb{N}_0$ verifying that $j+q\geq M$ one has that $a_{j,q}\leq \frac{\min\{a_{r,\ell}: r\leq m_0, \ell \leq m_0\}}{2}$. Now we choose $\rho> 1$ large enough so that the sequence $\{\rho^{k}\hspace{0.05cm} a_{q+k,q}: k\geq -q\}$ is strictly increasing up to $m$, for each $0\leq q\leq M$. This choice is possible because the condition: \begin{equation*}
						Q^{k+1}\hspace{0.04cm} a_{q+k+1,q} >Q^k \hspace{0.04cm}a_{q+k,q} 
					\end{equation*} for some $Q>0$, is equivalent to $Q>\frac{a_{q+k,q}}{a_{q+k+1,q}}$, and hence we only need to take $\rho>\max\{\frac{a_{q+k,q}}{a_{q+k+1,q}}: q\in\{0,1,...,M\}, k\in\{-q,-q+1,...,0,...,m\}\}$. This choice of $\rho>0$ implies that the maximum of the sequence $\{\rho^{k} \hspace{0.04cm}a_{q+k,q}: k\geq -q\}$ is only achieved in $k\geq m$. Denote $L_0:=\max\{\rho^{k} \hspace{0.05cm}a_{q+k,q}: q\in\mathbb{N}_0, k\geq -q\}$. Since $a_{0,1}>0$, we can also assume, without loss of generality, that 
				\begin{equation*}
						\rho>\frac{\max\{a_{r,\ell}: r, \ell\in\mathbb{N}_0\}}{a_{1,0}}
					\end{equation*} 
				By doing so, we guarantee that the following inequality holds: 
				\begin{equation}\label{eq.NN}
						L_0\geq \rho \hspace{0.07cm}a_{1,0} >\max\{a_{r,\ell}: r, \ell\in\mathbb{N}_0\}.
					\end{equation}
					Now we observe that for all $q\geq M$, it is not possible neither to have that $\rho^{k}\hspace{0.05cm} a_{q+k,q}=L_0$ with $k\leq 0$. Otherwise, since $\rho>1$ we would have for $k\leq 0$ that \begin{equation*}
						L_0 > \max\{a_{q+k,q}: q\in\mathbb{N}_0, \hspace{0.05cm}k\geq -q\}\geq \rho^{k}\hspace{0.04cm} a_{q+k,q}=L_0,
					\end{equation*} which obviously is a contradiction with \eqref{eq.NN}. Finally, we will see that it is not possible to have that $\rho^{k} \hspace{0.05cm}a_{q+k,q}=L_0$, with $q\geq M$ and $k\leq m$. Indeed, if we suppose that $\rho^{k} \hspace{0.05cm}a_{q+k,q}=L_0$, with $q\geq M$ and $k\leq m$, then we get 
				\begin{equation*}
						L_0=\rho^k \hspace{0.05cm}a_{q+k,q}\leq \rho^m\hspace{0.05cm} a_{q+k,q}<\rho^m \hspace{0.05cm}a_{m,0}\leq L_0
					\end{equation*} since $\rho>1$, which obviously is a contradiction and we are done.
				\end{proof}

				\begin{remark}\label{Remark 3}
					Switching the roles of $j$ and $q$ and choosing $\rho<1$ small enough in the proof of Proposition \ref{main}, we obtain that for every $\lambda>0$ and $m\in\mathbb{N}$, there is $g\in \mathcal{S}_\omega(\mathbb{R})$ for which the supremum in $p_\lambda(g)$ is only attained in $j,q\in\mathbb{N}_0$ verifying $q-j\geq m$. In particular, $q\geq m$.  
				\end{remark}
				As far as we know, there is no literature available about whether all $3-$tuples $(j,q,\lambda)\in\mathbb{N}_0\times \mathbb{N}_0\times \mathbb{R}_+$ verify that there is $f\in \mathcal{S}_\omega(\mathbb{R})$ such that $$
				p_\lambda(f)=\sup_{x\in\mathbb{R}} |x|^q |f^{(j)}(x)|  \exp\left(-\lambda \varphi_\omega^*\left(\frac{j+q}{\lambda}\right)\right),$$
				 not even when the space $\mathcal{S}_\omega(\mathbb{R})$ is a classical Gelfand-Shilov class. Let us state the problem explicitly, for which we only have a partial answer as to which $3-$tuples hold the property described above:
				
				\begin{remark}
					Let $\omega$ be any weight function (for instance, the Gevrey weight). Given $\lambda>0$, $j,q\in\mathbb{N}_0$, is there any   $f\in\mathcal{S}_\omega(\mathbb{R})$  satisfying
					\begin{equation*}
						p_\lambda(f)=\max_{x\in\mathbb{R}} |x|^q |f^{(j)}(x)|  \exp\left(-\lambda \varphi_\omega^*\left(\frac{j+q}{\lambda}\right)\right)?
					\end{equation*}
				\end{remark}
			
			 Other useful fact is the following result:
			
				\begin{prop}\label{remarkFourier}
					Let $\omega$ be a weight function. Let $a\in\mathbb{R}\setminus\{\pm 1,0\}$, $\psi(x):=a x$ and $\Phi(x)=\frac{x}{a}$, for  $x\in\mathbb{R}$. The following conditions are equivalent.
					\begin{enumerate}
						\item The composition operator $C_\psi:\mathcal{S}_{\omega}(\mathbb{R})\to \mathcal{S}_{\omega}(\mathbb{R})$ is topologizable.
						\item The composition operator $C_\Phi:\mathcal{S}_{\omega}(\mathbb{R})\to \mathcal{S}_{\omega}(\mathbb{R})$ is topologizable.
					\end{enumerate}.
				\end{prop}
				\begin{proof}
					The result follows by observing that for all $b\not=0, \eta\in\mathbb{R}$ one has the following equalities: 
					\begin{equation*}
							(\mathcal{F}(f(b\hspace{0.05cm}\bullet)))(\eta) =\int_\mathbb{R} e^{-i \eta x}f(bx)dx=\frac{1}{b}\int_\mathbb{R}e^{-i\eta \frac{y}{b}}f(y)dy =\frac{1}{b} (\mathcal{F}f)\left(\frac{\eta}{b}\right),
					\end{equation*} 
					where the Fourier transform $\mathcal{F}: \mathcal{S}_{\omega}(\mathbb{R})\to \mathcal{S}_{\omega}(\mathbb{R})$ is an isomorphism onto. Accordingly, it follows that  $\mathcal{F} \circ C_{\psi_m}=(\frac{1}{a^m}C_{\Phi_m})\circ \mathcal{F}$ for all $m\in\mathbb{N}$. Putting these facts together we can easily conclude.
				\end{proof}
			
				We know that composition operators of non-trivial dilatations are not power bounded (see \cite[Proposition 3.1]{RACSAM}). Contrary to what was expected from the case of translations worked out above, composition operators of non-trivial dilatations are not topologizable on Gelfand-Shilov classes. Surprisingly enough, it depends on the possibility of finding $f\in \mathcal{S}_{\omega}(\mathbb{R})$ whose semi-norm $p_\lambda(f)$ is attained in $j,q\in\mathbb{N}_0$ verifying conditions like the one appearing in Proposition \ref{main}. The following result, whose proof is very different from the techniques used in \cite{jmaa, RACSAM}, is also unexpected from the classical Schwartz class $\mathcal{S}(\mathbb{R})$ (see \cite[Example 4.15.]{akj}).

				\begin{theorem}\label{Ctopol} Let $\omega$ be a weight function and
					$\psi(x):=a x$, for all $x\in\mathbb{R}$, with $a\not=0$. If the weight function $\omega$ satisfies condition $(\zeta)$, then the following conditions are equivalent.
					\begin{enumerate}
						\item $C_\psi\colon \mathcal{S}_\omega(\mathbb{R})\to \mathcal{S}_\omega(\mathbb{R})$ is power bounded.
						\item $C_\psi\colon \mathcal{S}_\omega(\mathbb{R})\to \mathcal{S}_\omega(\mathbb{R})$ is $m$-topologizable.
						\item  $C_\psi\colon \mathcal{S}_\omega(\mathbb{R})\to \mathcal{S}_\omega(\mathbb{R})$ is topologizable.
						\item $\{C_{\psi_m}: \mathcal{S}_{\omega}(\mathbb{R})\to \mathcal{S}(\mathbb{R}): m\in\mathbb{N}\}$ 
						is equicontinuous.
						\item $a=\pm 1$.
					\end{enumerate} 
				\end{theorem}
				\begin{proof} 
					Clearly, $1.\Rightarrow 2.\Rightarrow 3.$ By \cite[Proposition 3.1]{RACSAM}  we have that $1.\Leftrightarrow 4.\Leftrightarrow 5.$. So, to conclude the proof,  it suffices to show that if $\psi(x)=a x$,  for  $x\in\mathbb{R}$, with $a\not=\pm 1$, then the composition operator $C_\psi: \mathcal{S}_{\omega}(\mathbb{R})\to \mathcal{S}_{\omega}(\mathbb{R})$ is not topologizable. To this end, we first observe that
					$\psi_m(x)=a^m x$, for all $x\in\mathbb{R}$ and $m\in\mathbb{R}$.  
				 \\

					Now, we assume that the composition operator $C_\psi: \mathcal{S}_{\omega}(\mathbb{R})\to \mathcal{S}_{\omega}(\mathbb{R})$ is topologizable. By Proposition \ref{remarkFourier}, we may also assume that $|a|>1$ and hence, $\frac{1}{|a|^{m q}}<1$ for all $m, q\in\mathbb{N}$. Therefore, 
					for every $k\in \mathbb{N}$ there exist $h\in \mathbb{N}$ with $h\geq k$ and a sequence of positive constants $\{C_m\}_{m\in\mathbb{N}}\subset \mathbb{R}_+$ such that 
					\begin{equation*}
						p_k(C_{\psi_m}f)\leq C_m p_h(f)
					\end{equation*} 
					for every $f\in \mathcal{S}_{\omega}(\mathbb{R})$ and $m\in\mathbb{R}$, i.e., applying  Lemma \ref{lemmaBeurling}, 
					\begin{align*}
						p_k(C_{\psi_m} f)
					&=\sup_{j,q\in\mathbb{N}_0}\sup_{x\in\mathbb{R}} |x|^q |f^{(j)}(a^m x)| |a|^{m j} \exp\left(-k\varphi^*_\omega\left(\frac{j+q}{k}\right)\right) \\
						& \leq C_m p_h(f)=C_m \hspace{0.02cm}\sup_{y\in\mathbb{R}} |y|^{\overline{q}} |f^{(\overline{j})}(y)| \exp\left(-h\varphi^*_\omega\left(\frac{\overline{j}+\overline{q}}{h}\right)\right), 
					\end{align*} 
					where $\overline{j}, \overline{q}\in\mathbb{N}_0$ depend only on $f$ and $h\in\mathbb{N}$. 
					On the other hand, we also have that
					\begin{align*}
							p_k(C_{\psi_m} f)
						&= \sup_{j,q\in\mathbb{N}_0} \sup_{y\in \mathbb{R}} |y|^q |a|^{m(j-q)} |f^{(j)}(y)|  \exp\left(-k\varphi^*_\omega\left(\frac{j+q}{k}\right)\right) 
						\\&\geq |a|^{m(\overline{j}-\overline{q})} \frac{\exp(-k\varphi^*_\omega(\frac{\overline{j}+\overline{q}}{k})) }{\exp(-h\varphi^*_\omega(\frac{\overline{j}+\overline{q}}{h})) } p_h(f),
					\end{align*}
					for every $f\in \mathcal{S}_{\omega}(\mathbb{R})$ and $m\in\mathbb{N}$. Combining the inequalities above, it follows   that
					\begin{equation}\label{ast}
						|a|^{m(\overline{j}-\overline{q})} \frac{\exp(-k\varphi^*_\omega(\frac{\overline{j}+\overline{q}}{k})) }{\exp(-h\varphi^*_\omega(\frac{\overline{j}+\overline{q}}{h})) }\leq C_m
					\end{equation} 
					for every $f\in \mathcal{S}_{\omega}(\mathbb{R})$ and $m\in\mathbb{N}$, where $\overline{j}, \overline{q}\in\mathbb{N}_0$ depend only on $f$ and $h\in\mathbb{N}$. We observe that \eqref{ast} is equivalent in turn to 
					\begin{equation*}
						m(\overline{j}-\overline{q})\log(|a|)+h\varphi^*_\omega\left(\frac{\overline{j}+\overline{q}}{h}\right)-k\varphi^*_\omega\left(\frac{\overline{j}+\overline{q}}{k}\right) \leq \log(C_m),
					\end{equation*}
					for all $m\in\mathbb{N}$, where $\overline{j}, \overline{q}\in\mathbb{N}_0$ depend only on $f$ and $h\in\mathbb{N}$, as it easy to prove. \\

					Since $\omega$ is a weight function satisfying condition $(\zeta)$, we have that $\mathcal{S}_\omega(\mathbb{R})=\mathcal{S}_M(\mathbb{R})$, \cite[Corollary 16]{BMM}, for some weight sequence $M=(M_p)_{p\in\mathbb{N}_0}$. Accordingly, using the system $\{p_{M,h}\}_{h>0}$ of norms generating the topology of  $\mathcal{S}_M(\mathbb{R})$, the inequality
					\eqref{ast} is equivalent to 
					\begin{equation*}
						|a|^{m(\overline{j}-\overline{q})} \frac{A^{\overline{j}+\overline{q}}}{B^{\overline{j}+\overline{q}}} \leq C_m 
					\end{equation*} for all $m\in\mathbb{N}$, and for some $A,B>0$.
					This is in turn equivalent to
					\begin{equation}\label{ast1}
						\left( m\log(|a|)+\log\left(\frac{A}{B}\right)\right)\overline{j}-\overline{q}\left(m\log(|a|)-\log\left(\frac{A}{B}\right)\right)\leq \log(C_m)
					\end{equation}
					for all $m\in\mathbb{N}$. 
					
					By Proposition \eqref{main}, we can construct sequences $(f_\ell)_\ell\subset \mathcal{S}_{\omega}(\mathbb{R})=\mathcal{S}_M(\mathbb{R})$, $(j_\ell)_\ell\subset \mathbb{N}_0, (q_\ell)_\ell\subset\mathbb{N}$ such that the supremum involving $p_{M,k}(f_\ell)$ is attained for $j_\ell$ and $q_\ell$ satisfying $j_\ell-q_\ell\geq \ell$ for all $\ell\in\mathbb{N}$. In particular, $\frac{q_\ell}{j_\ell}\leq 1$ for all $\ell\in\mathbb{N}$. Let us assume that $\log|a|-\log\left(\frac{A}{B}\right)>0$ (the other case is similar and easier).
					Since by \eqref{ast1} we have that
					\begin{equation*}\label{astint}
						j_{\ell}\left[\left(m\log(|a|)+\log\left(\frac{A}{B}\right)\right)-\frac{q_{\ell}}{j_{\ell}}\left(\log(|a|)-\log\left(\frac{A}{B}\right)\right)\right]\leq \log(C_m),
					\end{equation*}
					for every $\ell, m\in\mathbb{N}$,  we get that
					\begin{equation}\label{eq.NNN}
		j_{\ell}\left((m-1)\log|a|+2\log\left(\frac{A}{B}\right)\right)\leq \log(C_m)
					\end{equation}
				for every $\ell, m\in\mathbb{N}$.

					 For any fixed $m\in\mathbb{N}$ large enough so that $(m -
					 1) \log |a| + 2\log\left(\frac{A}{B}\right)>0$, letting $\ell\to\infty$ we obtain  a contradiction because the left-hand side of  \eqref{eq.NNN} obviously tends to $+\infty$ as $\ell\to \infty$ and the right-hand side of  \eqref{eq.NNN} continues to be equal to $\log(C_m)\in\mathbb{R}$.
				\end{proof}

				In the following result we will show that the composition operator $C_\psi: \mathcal{S}_{\omega}(\mathbb{R})\to \mathcal{S}_{\sigma}(\mathbb{R})$, associated with a dilation, is topologizable for   certain weight function $\sigma$  such that $\mathcal{S}_{\omega}(\mathbb{R})\subset \mathcal{S}_{\sigma}(\mathbb{R})$.
				\begin{prop}\label{dilationdelta} Let $\omega$ be a weight function and
					$\psi(x)=a x$, for $x\in\mathbb{R}$, with $a\not=0$. Then $\{C_{\psi_m}:\mathcal{S}_{\omega}(\mathbb{R})\to \mathcal{S}_{\omega(\bullet^{\frac{1}{1+\delta}})}(\mathbb{R}): m\in\mathbb{N}\}$ is topologizable, for every $\delta>0$.
				\end{prop}
				\begin{proof}
					%
					First,  suppose  that $|a|>1$ and fix $\lambda>0, \delta>0$.
					Since $\lim_{j\to \infty}\frac{|a|^{m j}}{j!^\delta}=0$ for every $m\in\mathbb{N}$ and that for all $D>0, \mu>0$ there is $B>0$ such that \begin{equation*}
						D^j j! \leq B \exp\left(\mu \varphi^*_\omega\left(\frac{j}{\mu}\right)\right)
					\end{equation*} 
					for all $j\in\mathbb{N}_0$, we obtain that for each $m\in\mathbb{N}$ there is $D_{\lambda,\delta,m}>0$ such that
					\begin{equation*}
						|a|^{m j} \leq D_{\lambda,\delta,m} \exp\left(\lambda \varphi^*_\omega\left(\frac{\delta j}{\lambda}\right)\right)
					\end{equation*}
					for all $j\in \mathbb{N}_0$. 
					
					In view of the inequality above, we obtain for every $f\in \mathcal{S}_{\omega}(\mathbb{R})$ that
					\begin{equation*}
						\begin{split}
							|x|^q |(f\circ \psi_m)^{(j)}(x)| &=|a^m x|^q |f^{(j)}(a^m x)| |a|^{m(j-q)} \\& \leq 
							p_{\omega,\lambda}(f) \exp\left(\lambda \varphi^*_\omega\left(\frac{j+q}{\lambda}\right)\right) |a|^{mj} \\ &\leq D_{\lambda,\delta,m} \exp\left(\lambda \varphi^*_\omega\left(\frac{(1+\delta)(j+q)}{\lambda}\right)\right)p_{\omega,\lambda}(f)
						\end{split}
					\end{equation*} 
					for all $x\in\mathbb{R}$, $q\in\mathbb{N}_0$, $j\in\mathbb{N}$ and $m\in\mathbb{N}$. Since $\lambda>0$ and $\delta>0$ are arbitrary, we have done with the case $|a|>1$. We similarly deal with the case $0<|a|<1$. 
				\end{proof}
				
				\begin{remark} We can conclude that 
					given a weight function $\sigma$, the composition operator having some dynamic property (such as topologizability) for every Gelfand-Shilov class $\mathcal{S}_{\sigma_\delta}(\mathbb{R})$, with $\delta>1$ and $\sigma_\delta(t)=\omega(t^{\frac{1}{\delta}})$, is not enough to state that such a dynamic property is also valid in the strictly smaller Gelfand-Shilov class $\mathcal{S}_{\sigma}(\mathbb{R})$.
				\end{remark}
				
				\begin{remark}
 Let $p>1$ and $\omega(t)=\left(\max\{0,\log(t)\}\right)^p$, for $t\geq 0$. It is well-known (see \cite[20. Example]{BMM}) that the weight $\omega$ does not verify condition $(\zeta)$. Since $\mathcal{S}_{\sigma_\delta}(\mathbb{R})=\mathcal{S}_{\omega}(\mathbb{R})$, for every $\delta>1$, where $\sigma_\delta(t)=\omega(t^{\frac{1}{\delta}})$, if $\psi(x)=a \hspace{0.04cm}x$, for $x\in \mathbb{R}$, with $a\not =0$, by Proposition \ref{dilationdelta}, the corresponding $C_\psi: \mathcal{S}_{\omega}(\mathbb{R})\to \mathcal{S}_{\omega}(\mathbb{R})$ is topologizable.
				\end{remark}
				
				\medskip
				\noindent
				\textbf{Acknowledgement} The authors would like to thank the Referee for her/his useful
				comments and suggestions.

				\medskip
				\noindent
				\textbf{Data Availability Statement} Not applicable
				\medskip

				\noindent
				{\textbf{Declarations }}
				
				\medskip
				\noindent
				\textbf{Conflict of interest}
				 The authors have no Conflict of interest to declare that	are relevant to the content of this article.

\end{document}